\documentclass[runningheads]{llncs}

\usepackage[T1]{fontenc}
\usepackage[utf8]{inputenc}

\usepackage{graphicx}
\usepackage[caption=false]{subfig}
\usepackage{hyperref}

\usepackage{xcolor}
\hypersetup{
    colorlinks,
    linkcolor={red!50!black},
    citecolor={blue!50!black},
    urlcolor={blue!80!black}
}

\usepackage{amsmath, amsfonts}
\usepackage[shortlabels,inline]{enumitem}

\setlist[enumerate]{leftmargin=1.2cm}

\usepackage{todonotes}
\newtheorem{observation}{Observation}

\newcommand{\steven}[1]{\textcolor{black}{#1}}

\newcommand{\sknov}[1]{\textcolor{black}{#1}}

\newenvironment{hidden}{}

\usepackage{comment}
\excludecomment{hidden}
\usepackage{tikz}

\usetikzlibrary{decorations.pathreplacing}
\newcommand{\vertex}{\node[vertex]}
\tikzstyle{vertex}=[draw, shape=circle, minimum size=0.5em, inner sep=1, fill]

\title{Snakes and Ladders: a Treewidth Story\thanks{
Corresponding author: S. Kelk. R.
Meuwese was supported by the Dutch Research Council (NWO) KLEIN 1 grant \emph{Deep kernelization for phylogenetic discordance}, project number OCENW.KLEIN.305. \sknov{A preliminary version of this article appeared in the
proceedings of WG 2023 (Workshop on Graph-Theoretic Concepts in Computer Science)~\cite{chaplick2023snakesWG} and as an earlier ArXiv upload (February 2023)~\cite{chaplick2023snakes}.} }}
\author{Steven Chaplick\inst{1} \and Steven Kelk\inst{1} \and Ruben Meuwese\inst{1} \and Mat\'u\v{s} Mihal\'ak\inst{1} \and Georgios Stamoulis\inst{1}}
\institute{Dept. Advanced Computing Sciences, Maastricht University, the Netherlands \email{ \{s.chaplick,steven.kelk,r.meuwese,matus.mihalak,\\georgios.stamoulis\}@maastrichtuniversity.nl}}
\authorrunning{S. Chaplick et al.}

\begin{document}

\maketitle

\begin{abstract}
Let $G$ be an undirected graph. We say that $G$ contains a ladder of length $k$ if the $2 \times (k+1)$ grid graph is an induced subgraph of $G$ that is only connected to the rest of $G$ via its four cornerpoints. 
We prove that if all the ladders contained in $G$ are reduced to length 4, the treewidth remains unchanged (and that this bound is tight). 
Our result indicates that, when computing the treewidth of a graph, long ladders can simply be reduced, and that minimal forbidden minors for bounded treewidth graphs cannot contain long ladders. 
Our result also settles an open problem from algorithmic phylogenetics: the common chain reduction rule, used to simplify the comparison of two evolutionary trees, is treewidth-preserving in the display graph of the two trees. 
%
%We conclude with a discussion of how our results might be generalized to other ladder-like structures.
\end{abstract}


\section{Introduction}
This is a story about treewidth, but it starts in the world of biology. 
A phylogenetic tree on a set of leaf labels $X$
%, representing a set of contemporary species,
is a binary tree representing the evolution of $X$. These are studied extensively in computational biology~\cite{steel2016phylogeny}. 
Given two such trees a natural aim is to quantify their topological dissimilarity~\cite{john2017shape}. Many such dissimilarity measures have been devised and they are often NP-hard to compute, stimulating the application of techniques from parameterized complexity~\cite{bulteau2019parameterized}.  
Recently there has been a growing focus on treewidth. 
This is because, if one takes two phylogenetic trees on $X$ and identifies leaves with the same label, we obtain an auxiliary graph structure known as the \emph{display graph}~\cite{bryant2006compatibility}. Crucially, the treewidth of this graph is often bounded by a function of the dissimilarity measure that we wish to compute~\cite{kelk2016monadic}. This has led to the use of Courcelle's Theorem within phylogenetics (see e.g. \cite{janssen2018treewidth,kelk2016monadic}) and explicit dynamic programs running over tree decompositions; see \cite{van2022embedding} and references therein. 
In~\cite{kelk2017treewidth} the spin-off question was posed: is the treewidth of the display graph actually a meaningful measure of phylogenetic dissimilarity \emph{in itself} - as opposed to purely being a route to efficient algorithms? 
A closely-related question was whether parameter-preserving reduction rules,  applied to two phylogenetic trees to shrink them in size, also preserve the treewidth of the display graph?
The well-known \emph{subtree reduction rule} is certainly treewidth preserving~\cite{kelk2017treewidth}. 
However, the question remained whether the \emph{common chain reduction rule}~\cite{AllenSteel2001} is treewidth-preserving. 
A common chain is, informally, a sequence of leaf labels $x_1, \ldots, x_k$ that has the same order in both trees.
Concretely, the question arose~\cite{kelk2017treewidth}: is it possible to truncate a common chain to \emph{constant length} such that the treewidth of the display graph is preserved? 
Common chains form ladder-like structures in the display graph, i.e., this question is about how far ladders can be reduced in length without causing the treewidth to decrease.

In this article we answer this question affirmatively, and more generally. 
Namely, we do not restrict ourselves to display graphs, but consider arbitrary graphs. 
A \emph{ladder $L$ of length} $k\geq 1$ of a graph $G$ is a $2 \times (k+1)$ grid graph such that $L$ induces (only) itself and that $L$ is only connected to the rest of the graph by its four cornerpoints. 
First, we prove that a ladder $L$ can be reduced to length 4 without causing the treewidth to decrease, and that this is best possible: reducing to length 3 sometimes causes the treewidth to decrease. 
We also show that if $tw(G) \geq 4$ then reduction to length 3 is safe and,  again,  best possible. 
\sknov{These tight examples are also shown to exist for 
all higher treewidths: for every $t \geq 4$ we can construct a graph $G$ with treewidth $t$ and a ladder of length 3, such that shortening the ladder to length 2 causes the treewidth to decrease to $t-1$.}
%higher treewidths. 
Returning to phylogenetics, and thus when $G$ is a display graph, we leverage the extra structure in these graphs to show that common chains can be reduced to 4 leaf labels (and thus the underlying ladder to length 3) without altering the treewidth: this result is thus slightly stronger than on general~$G$.

Our proofs are based on first principles: we directly modify a tree decomposition to get what we need. In doing so we come across the problem that, unless otherwise brought under control, the set of bags that contain a given ladder vertex of $G$ can wind and twist through the tree decomposition in very pathological ways. Getting these \emph{snakes} under control is where much of the hard work and creativity lies, and is the inspiration for the title of this paper. 

From a graph-theoretic perspective our results have the following significance.
First, it is standard folklore that shortening paths (i.e. suppressing vertices of degree 2) is treewidth-preserving, but there is seemingly little in the literature about shortening recursive structures that are slightly more complex than paths, such as ladders. 
(Note that Sanders~\cite{sanders1996linear} did consider ladders, but only for recognizing graphs of treewidth at most 4, and in such a way that the reduction destroys the ladder topology).
Second, our results imply a new safe reduction rule for the computation of treewidth; a survey of other reduction rules for treewidth can be found in \cite{Abu-Khzam2022}. 
Third, we were unable to find sufficiently precise  machinery, characterisations of treewidth or restricted classes of tree decomposition in the literature that would facilitate our results.
Perhaps most closely related to our ladders are the more general \emph{protrusions}: low treewidth subgraphs that ``hang'' from a small boundary~\cite[Ch.~15-16]{kernelization-book_2019}.
There are general (algorithmic) results~\cite{BodlaenderFLPST16} wherein one can safely cut out a protrusion and replace it with a graph of parameter-proportional size instead -- these are based on a problem having \emph{finite integer index}~\cite{BodlaenderF01}. 
Such techniques might plausibly be used to prove that there is \emph{some} constant to which ladders might safely be shortened, but our tight bounds seem out of their reach.
Finally, the results imply that minimal forbidden minors for bounded treewidth cannot have long ladders.

We conclude the article with a number of auxiliary insights and a discussion of possible future work.\\
\\
\emph{Author's note 1}. Compared to the earlier version of this article~\cite{chaplick2023snakesWG,chaplick2023snakes} the present version includes analytical, as opposed to empirical, tightness results (Lemma \ref{lem:alwaystight}) and an extended discussion.\\
\\
\emph{Author's note 2}. At the time of writing, January 2024, we learned that Theorem \ref{thm:main3} was already proven in the article \cite[Theorem 2]{almob2023} (based on the earlier WABI 2022 article \cite{wabi2022}) using different techniques. Our
Theorem \ref{thm:main} can be viewed as a very mild strengthening of  \cite[Theorem 2]{almob2023} in the case that $tw(G) \geq 4$, although the proof of \cite[Theorem 2]{almob2023} could be adapted to this end. Our Theorem \ref{thm:preserve2} strengthens \cite[Theorem 2]{almob2023} to the case of display graphs. Our Lemma \ref{lem:alwaystight}, in which tightness of ladder bounds is proven for all treewidths, is unique to our article.

\section{Preliminaries} 
We follow \cite{kelk2017treewidth} for notation.
A \emph{tree decomposition} of an undirected graph $G=(V,E)$ is a pair $(\mathcal{B}, \mathbb{T})$ where $\mathcal{B} = \{B_1, \dots ,B_q\}$, $B_i \subseteq V(G)$, is a multiset of \emph{bags}
and $\mathbb{T}$ is a tree with $q$ nodes, which are in bijection with $\mathcal{B}$, and satisfy the following three properties
\begin{enumerate}
    \item[(tw1)] $\cup_{i=1}^q B_i = V(G)$;
    \item[(tw2)] $\forall e = \{ u,v \} \in E(G), \exists B_i \in \mathcal{B} \mbox{ s.t. } \{u,v\} \subseteq B_i$;
    \item[(tw3)] $\forall v \in V(G)$, all the bags $B_i$ that contain $v$ form a connected subtree of~$\mathbb{T}$.
\end{enumerate}

The \emph{width} of $(\mathcal{B}, \mathbb{T})$ is equal to $\max_{i=1}^q |B_i|-1$. The \emph{treewidth} of $G$, denoted $tw(G)$, is the smallest width among all tree decompositions of $G$.
%For any non-trivial connected graph $G$,  $tw(G)=1$ if and only if $G$ is a tree.
Given a tree decomposition $\mathbb{T}$ of a graph $G$, we denote by $V(\mathbb{T})$ the (multi)set of its bags and
by $E(\mathbb{T})$ the set of its edges. Property (tw3) is also known as \emph{running intersection property}.
%
%It is well known that for disconnected $G$ the treewidth of $G$ is simply the maximum ranging over the treewidth of all its connected components; for this reason we will henceforth always assume that $G$ is connected. 
Without loss of generality, we consider only connected graphs $G$.

Note that subdividing an edge $\{u,v\}$ of $G$ with a new degree-2 vertex $uv$ does not change the treewidth of $G$. 
%If $G$ is a tree it remains a tree; if $G$ is not a tree (and thus has treewidth at least 2) then, for any tree decomposition of $G$, we can add a bag $\{u,v,uv\}$ pendant to any bag containing $u$ and $v$, obtaining a tree decomposition of the same width as the original one.
In the other direction, suppression of degree-2 vertices is also treewidth preserving \emph{unless} it causes the only cycle in a graph to disappear (e.g. if $G$ is a triangle); unlike \cite{kelk2017treewidth} we will never encounter this boundary case. 
%We will never encounter this boundary case so we henceforth that edge subdivision and degree-2 vertex suppression do not impact upon the treewidth.
An equivalent definition of treewidth is based on chordal graphs. Recall that a graph $G$ is chordal if every induced cycle in $G$ has exactly three vertices. The treewidth of $G$ is the minimum, ranging over \emph{all} chordal completions $c(G)$ of $G$ (we add edges until $G$ becomes chordal), of the  size of the maximum clique
in $c(G)$ minus one. Under this definition, each bag of a tree decomposition of $G$ naturally corresponds to a maximal clique in a chordal completion of $G$ \cite{Blair1993}.

% In the following, the bags of tree decompositions will be referred as \textit{nodes} to avoid
%confusion with \textit{vertices} of the underlying graphs or trees.

%For a graph $G=(V,E)$ and an edge $e = \{u,v\} \in E(G)$, the \emph{deletion} of $e$ is the operation which simply deletes $e$ from $E(G)$ and leaves the rest of the graph $G$
%the same. The \emph{contraction} of $e$, denoted $G/e$, is the operation where edge $e$ is deleted and its incident vertices $u,v$ are identified.

% i.e., we replace $u,v$ with a
%new, single vertex such that edges incident to this new vertex are precisely these edges, other %than $e$, that are incident to one of $u,v$ in $G$.
We say that a graph $H$ is a \emph{minor} of another graph $G$ if $H$ can be obtained 
from $G$ by deleting edges and vertices and by contracting edges.

A \emph{ladder $L$ of length $k \geq 1$} is a $2 \times (k+1)$ grid graph. A \emph{square} of $L$ is a set of vertices of $L$ that induce a 4-cycle in $L$. We call the endpoints of $L$, i.e., the degree-2 vertices of $L$, the \emph{cornerpoints} of $L$.
%
%we call the $k$ chordless cycles of length $4$ in $L$ its \emph{squares}, and the degree-2 vertices its \emph{cornerpoints}.% for obvious reasons.  
We say that a graph $G$ \emph{contains} $L$ if the following holds (see Fig.~\ref{fig:ladder} for illustration):
\begin{enumerate}
    \item The subgraph induced by vertices of $L$ is $L$ itself.
    %i.e. there are no edges between vertices of $L$ other than those of the grid graph.
    \item Only cornerpoints of $L$ can be incident to an edge with an endpoint outside~$L$.
%    \item Apart from possibly the four cornerpoints of $L$, there are no edges in $G$ with one endpoint in $L$ and one outside $L$.
\end{enumerate}
%See Figure \ref{fig:ladder} for illustration. 
%Condition 2 allows the possibility that each of the four cornerpoints has one, multiple or zero edges incident to it whose other endpoint lies outside the ladder.

Observe that a ladder of length $k$ is a minor of the ladder of length $(k+1)$. Treewidth is non-increasing under the action of taking minors, so reducing the length of a ladder in a graph cannot increase the treewidth of the graph.

Suppose $G$ contains a ladder $L$. We say that $L$ \emph{disconnects} $G$ if $L$ contains a square $\{u,v,w,x\}$ such
that the two horizontal edges of the square (following Fig.~\ref{fig:ladder}, these are the edges
$\{u,w\}$ and $\{v,x\}$) form an edge cut of the entire graph $G$.
%deleting the two horizontal edges of the square (following Figure \ref{fig:ladder}, these are the edges
%$\{u,w\}$ and $\{v,x\}$) disconnects $G$ into exactly two connected components. 
Note that a square of $L$ has this property if and only if all squares of $L$ do.
Also, if we reduce the length of a ladder $L$ to obtain a shorter ladder $L'$, $L'$ disconnects $G$ if and only if $L$ does.
We recall a number of results from Section 5.2 of \cite{kelk2017treewidth}; these will form the starting point for our work.
%We start with a straightforward case when $G$ contains a disconnecting ladder $L$. 
%A strong result can be proven easily when $L$ contains a ladder that disconnects $G$.
\begin{lemma}[\cite{kelk2017treewidth}]
%\cite[Lemma 5.1]{kelk2017treewidth}]
\label{lem:disconnecting}
 Suppose $G$ contains a disconnecting ladder $L$. % of length $1$ or longer.
 The ladder $L$ can be increased arbitrarily in length without increasing the treewidth of $G$.
\end{lemma}
For the more general %situation (i.e. $L$ is not necessarily disconnecting), the following weaker result is known.
case, the following weaker result is known.
\begin{lemma}[\cite{kelk2017treewidth}]
\label{lem:protoplusone}
%\cite[Lemma 5.2]{kelk2017treewidth} 
Suppose $G$ has $tw(G) \geq 3$ and contains a ladder. % $L$ with length 1 or longer. 
If the ladder is increased arbitrarily in length, the treewidth of $G$ increases by at most one.
\end{lemma}
%Now, consider the following (new) observation.
We now make the following (new) observation.
\begin{observation}
\label{obs:k4}
Suppose $G$ contains a ladder $L$ of length $2$ or longer. If $L$ is not disconnecting, then $tw(G) \geq 3$.
\end{observation}
\begin{proof}
Let $a,b,u,v,w,x$ be the six vertices in the ladder $L$ with edges as shown in Fig.~\ref{fig:ladder}. Observe that $G$ contains a $K_4$ minor. Specifically,
take $\{a,b\}, u, v$ and $\{w,x\}$ as the four corners of the minor pre-contraction. The fact that $L$ is not disconnecting means that there is a path from $\{a,b\}$ to $\{w,x\}$ that leaves the ladder at one end and re-enters it via the other, inducing the minor edge between $\{a,b\}$ and $\{w,x\}$. $K_4$ is the (unique) forbidden minor for graphs of treewidth at most 2, so $tw(G)>2$.
\end{proof}

We can leverage Observation \ref{obs:k4} to reformulate Lemma \ref{lem:protoplusone} without the $tw(G) \geq 3$ assumption. However it then only applies to ladders of size at least two.
%the starting length of the ladder $L$ in the statement of the lemma increases from 1 to 2.
\begin{lemma}
\label{lem:plusone}
Suppose $G$ contains a ladder $L$ with length at least 2. % or longer. 
If %the ladder 
$L$ is increased arbitrarily in length, the treewidth of the graph increases by at most one.
\end{lemma}
\begin{hidden}
\begin{proof}
If $L$ is disconnecting then Lemma \ref{lem:disconnecting} applies. If $L$ is not disconnecting, then by Observation \ref{obs:k4} we have $tw(G) \geq 3$ and Lemma \ref{lem:protoplusone} applies.
\end{proof}
\end{hidden}
If we start from a sufficiently long ladder, can the ladder be increased in length without increasing the treewidth? Past research has the following partial result. %This led to the following theorem.

\begin{theorem}[\cite{kelk2017treewidth}] 
%\cite[Theorem 5.2]{kelk2017treewidth}
\label{thm:unbounded}
Let $G$ be a graph with $tw(G)=k$. There is a value $f(k)$ such that if 
%there exists a ladder 
$G$ contains a ladder of length $f(k)$ or longer, 
the ladder can be increased in length arbitrarily without altering (in particular: increasing) the treewidth. 
\end{theorem}
%The proof of Theorem \ref{thm:unbounded} is based on forbidden minors for bounded treewidth graphs.
%The function $f$ is highly exponential in $k$ but well-defined. The theorem does not give us what we need, however, because
Ideally we would like a single, universal value \emph{that does not depend
on} $k$. In this article we will show that such a single, universal constant does exist.
%Nevertheless, it does have a role to play later on, as we shall see.
%
%
\section{Results}
We first consider graphs of treewidth at least 4; we later remove this restriction.
\begin{theorem}
\label{thm:main}
Let $G$ be a graph with $tw(G)\geq 4$. If $G$ has a ladder $L$ of length 3 or higher, the ladder can
be lengthened arbitrarily without changing  %(in particular: increasing) 
the treewidth.
%If $G$ has a ladder longer than length 3, the ladder can be reduced to length $3$ without
%altering (in particular: lowering) the treewidth. 
\end{theorem}

\begin{proof}
%Let $L$ be a ladder of %length $\ell > 3$.
Due to Lemma \ref{lem:disconnecting} we 
%already proves a stronger result for the case when $L$ is disconnecting.
%(a ladder with 1 square can be arbitrarily lengthened). 
can assume that $L$ is not disconnecting.  Our general strategy is to show that if $G$ contains the ladder $L$ shown in Fig.~\ref{fig:ladder}, we can insert an extra `rung' in the ladder without increasing the treewidth, thus obtaining a ladder with one extra square (see Fig.~\ref{fig:ladder2}).
%We will not always introduce the new edge exactly in the square suggested by Fig.~\ref{fig:ladder2} but due to isomorphism the end result will be the same.
The extension of the ladder by one square can then be iterated to obtain an arbitrary length ladder. %This will prove the result.
\begin{figure}[t]
\begin{minipage}[b]{0.44\textwidth}
    \centering
    \includegraphics[scale=1.0]{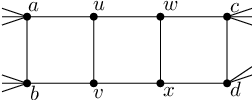}
    \caption{A ladder $L$ of length 3 with corner points $a,b,c,d$. 
    %Vertices $a,b,c,d$ are the corner points of $L$.  are each drawn with one neighbour outside $L$ but each of the four cornerpoints is permitted to have one, many or zero edges connecting it to the rest of the graph.
    }
    \label{fig:ladder} 
\end{minipage}
\hfill
\begin{minipage}[b]{0.48\textwidth}
    \centering
    \includegraphics[scale=1.0]{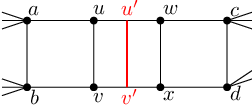}
    \caption{Inserting a new edge $\{u',v'\}$ into ladder $L$ results in ladder $L'$ of length 4.}
    %A ladder $L'$ of length 4, obtained by inserting an extra edge $\{u',v'\}$ into a ladder of length 3. Note that the extra edge can be added in any square.
    \label{fig:ladder2} 
\end{minipage}
\end{figure}

Let $L$ be the ladder shown in Fig.~\ref{fig:ladder}, and assume that $G$ contains $L$. 
Let $(\mathcal{B}, \mathbb{T})$ be a minimum-width tree decomposition for $G$.
We proceed with a case analysis. The cases are cumulative: we will assume that earlier cases do not hold.

\smallskip

\noindent
\textbf{Case 1.} \textbf{Suppose that $\mathcal{B}$ contains a bag $B$ such that all four vertices from one of the squares of $L$ are in $B$.} 
Let $\{u,v,w,x\}$, say, be the square of $L$ contained in bag $B$, where the position of the vertices is as in Fig.~\ref{fig:ladder}. We prolong the ladder as in Fig.~\ref{fig:ladder2} and create a valid tree decomposition for the new graph as follows:
we introduce a new size-5 bag $B'=\{u',u,v,w,x\}$ which we attach pendant to $B$ in the tree decomposition, and a new size-5 bag $B''=\{u',v',v,w,x\}$ which we attach pendant to $B'$. Observe that this is a valid tree decomposition for the new graph. Due to the fact that $tw(G) \geq 4$, the treewidth does not increase, and the statement follows. 
Note that in this construction $B''$ contains all four of $\{u',w,v',x\}$, which is a square of the new ladder, so the construction can be applied iteratively many times as desired to produce a ladder of arbitrary length.

%Assume without loss of generality that $\{u,v,w,x\} \subseteq B$, where $\{u,v,w,x\}$ is a square of $L$. In this case we introduce a new size-5 bag $B'=\{u',u,v,w,x\}$ which we attach pendant to $B$ in the tree decomposition, and a new size-5 bag $B''=\{u',v',v,w,x\}$ which we attach pendant to $B'$. Observe that this is valid tree decomposition for the graph obtained from $G$ by introducing the new edge as shown in Fig.~\ref{fig:ladder2}. Due to the fact that $tw(G) \geq 4$, the treewidth does not increase, and the statement follows. Note that in this construction $B''$ contains all four of $\{u',w,v',x\}$, which is a square of the new ladder, so the construction can be applied iteratively many times as desired to produce a ladder of arbitrary length.\\

\smallskip

\noindent
\textbf{Case 2.} \textbf{Suppose that $\mathcal{B}$ contains
a bag $B$ such that $|B \cap \{a,u,w,c\}| \geq 2$ and  $|B \cap \{b,v,x,d\}| \geq 2$.}
%This case is actually a generalization of Case 1.
Let $h_1, h_2$ be two distinct vertices
from $B \cap \{a,u,w,c\}$ and $l_1, l_2$ be two distinct vertices
from $B \cap \{b,v,x,d\}$.
%Note that, due to the assumption that Case 1 does not hold, it is not possible that $\{h_1, h_2, l_1, l_2\}$ is a square.

Observe that it is possible to partition the sequence $a, u, w, c$ into two disjoint intervals $H_1, H_2$, and the sequence $b,v,x,d$ into two disjoint intervals $L_1, L_2$ such that $h_1 \in H_1$, $h_2 \in H_2$, $l_1 \in L_1$ and $l_2 \in L_2$. If we contract the edges and vertices in each of $H_1, H_2, L_1, L_2$ we obtain a new graph $G'$ which is a minor of $G$. Note that $G'$ is similar to $G$ except that the ladder now has two fewer squares -- the three original squares have been replaced by a square whose corners correspond to $H_1, H_2, L_1, L_2$. This square might contain a diagonal but we simply delete this. We have $tw(G') \leq tw(G)$ because treewidth is non-increasing under taking minors. Now, by projecting the contraction operations onto $(\mathcal{B}, \mathbb{T})$ in the usual way\footnote{In every bag of the decomposition vertices from $H_1$ all receive the vertex label $H_1$, and similarly for the other subsets $H_2, L_1, L_2$.}, we obtain a tree
decomposition $(\mathcal{B}', \mathbb{T}')$ for $G'$ such that the width of $\mathbb{T}'$ is less than or equal to the width of $\mathbb{T}$. The bag in $(\mathcal{B}', \mathbb{T}')$ corresponding to $B$, let us call this $B'$, contains all four vertices $H_1, H_2, L_1, L_2$. Clearly, $\mathbb{T'}$ is a valid tree decomposition for $G'$. We distinguish two subcases.
\begin{enumerate}
\item If $\mathbb{T}'$ has width at
least 4, we can repeatedly apply the Case 1 transformation to $B'$ to produce an arbitrarily long ladder without raising the width of $\mathbb{T}'$. The resulting decomposition will thus have width no larger than $\mathbb{T}$. %Hence we obtain a decomposition whose width is no larger than $\mathbb{T}$ for a graph with an arbitrarily long ladder.
\item Suppose $\mathbb{T}'$ has width strictly less than 4, and thus strictly less than the width of $\mathbb{T}$. The width of $\mathbb{T}'$ is at least 3 because of the bag containing $H_1, H_2, L_1, L_2$. Case 1 introduces size-5 bags and can thus raise the width of the decomposition by at
most 1. Hence we again obtain a decomposition whose width is no larger than $\mathbb{T}$ for a graph with an arbitrarily long ladder. 
\end{enumerate}
%We henceforth assume that neither Case 1 nor Case 2 holds. \todo{GS:We already said that the cases are cumulative}
This concludes Case 2.  Moving on, any chordalization of $G$ must add the diagonal $\{w,v\}$ and/or the diagonal $\{u,x\}$. 
%Hence, in any tree decomposition of $G$, at least one of the following situations holds: (1) There is a bag containing $\{u,w,v\}$ and another bag (not necessarily distinct) containing $\{v,w,x\}$, (2) there is a bag containing $\{u,w,x\}$ and another bag (not necessarily distinct) containing $\{u,v,x\}$.
Hence we can assume that there is a bag containing $\{u,w,v\}$ and another bag containing $\{v,w,x\}$. (If the other diagonal is added we can simply flip the labelling of the vertices in the horizontal axis i.e. $a \Leftrightarrow b, u \Leftrightarrow v$ and so on). As Case 1 does not hold we can assume that the bag containing $\{u,w,v\}$ is distinct 
from the bag containing $\{v,w,x\}$.

For the benefit of later cases we impose extra structure on our choice of minimum-width tree decomposition of $G$. 
The \emph{distance} of decomposition $(\mathcal{B}, \mathbb{T})$ is the minimum, ranging over all pairs of bags $B_1, B_2$ such that $B_1$ contains $\{u,w,v\}$ and $B_2$ contains $\{v,w,x\}$, of the length of the path
%(in terms of number of bags)
in $\mathbb{T}$ from~$B_1$~to~$B_2$. \\
%If the distance is 0, for example, then there exists a bag $B_1$ containing $\{u,w,v\}$ and an adjacent bag containing $\{v,w,x\}$.\\
\\
\fbox{
\parbox{0.97\textwidth}{
\textbf{We henceforth let $(\mathcal{B}, \mathbb{T})$ be a minimum-width tree decomposition of $G$ such that, ranging over all minimum-width tree decompositions, the distance is minimized. Clearly such a tree decomposition exists.}}}\\
\\
Let $B_1, B_2$ be two bags from $\mathcal{B}$ with $\{u,w,v\} \subseteq B_1$, $\{v,w,x\} \subseteq B_2$ which achieve this minimum distance. Let $P$ be the path of bags from $B_1$ to $B_2$, including $B_1$ and $B_2$. We assume that $P$ is oriented left to right, with $B_1$ at the left end and $B_2$ on the right. As Case 2 does not hold, we obtain the following.
%(excluding $B_1$ and $B_2$ themselves).

\begin{observation}
\label{obs:stuff_not_there}
$B_1$ does not contain $b$, $x$ or $d$, and $B_2$ does not contain $a$, $u$, $c$.
\end{observation}
\begin{hidden}
\begin{proof}
If this was not true then Case 2 would hold, contradiction.
\end{proof}
\end{hidden}

%\smallskip

\noindent
%\textbf{Case 3. $P$ contains no bags other than $B_1$ and $B_2$.} 
%That is, $B_1$ and $B_2$ are adjacent. 
\textbf{Case 3. $B_1$ and $B_2$ are adjacent in $P$.} Although this could be subsumed into a later case it introduces important machinery;
we therefore treat it separately.% This case is illustrative for some of the concepts we address later, so we will explain it in detail. (It could be subsumed into the case that $P$ does contain bags other than $B_1$ and $B_2$.)

\emph{Subcase 3.1:} Suppose $a \in B_1$ (or, completely symmetrically, $d \in B_2$). Note that in this case all the edges in $G$ incident to $u$ are covered by $B_1$. Hence, we can
safely delete $u$ from all bags except $B_1$. Next, we create a new bag $B^{*} = \{a,u,w,v\}$ and attach it pendant 
to $B_1$, and finally we replace $u$ with $x$ in $B_1$. It can be easily verified that this is a valid tree decomposition for $G$ and that the width is not increased, so it is still a minimum-width tree decomposition. However, $B_1$ is now a candidate for Case 2, and we are done. Note that replacing $u$ with $x$ in $B_1$ is only possible because $B_1$ is next to $B_2$ in $P$.

\emph{Subcase 3.2:} Suppose Subcase 3.1 does not hold. Then $a \not \in B_1$ (and, symmetrically, $d \not \in B_2$). Putting
all earlier insights together, we see $a, b, x, d \not \in B_1$ and $a,u,c,d \not \in B_2$. Observe that $a$, which is not in $B_2$, is not in any bag to the right of $B_2$. If it was, then the fact that some bag contains the edge $\{a,u\}$, and the running intersection property, entails that $B_2$ would contain at least one of $a$ and $u$, neither of which is permitted. Hence, if $a$ appears in bags other than $B_1$, they are all in the left part of the decomposition. Completely symmetrically, if $d$ is in bags other than $B_2$, they are all in the right part of the decomposition. Because of this, $b$ can only appear on the left of the decomposition (because the edge $\{a,b\}$ has to be covered) and $c$ can only be on the right of the decomposition (because of the edge $\{c,d\}$). Summarising, $B_1$ (respectively, $B_2$) does not contain $a$ or $b$ (respectively, $c$ or $d$) and all bags containing $a$ or $b$ (respectively, $c$ or $d$) are in the left (respectively, right) part of the decomposition. Note that $c \not \in B_1$. This is because edge $\{c,d\}$ has to be in some bag, and this must necessarily be to the right of $B_2$: but then running intersection puts at least one of $c,d$ in $B_2$, contradiction. Symmetrically, $b \not \in B_2$. So $a,b,c,d,x \not \in B_1$ and $a,b,c,d,u \not \in B_2$.

We now describe a construction that we will use extensively:  \emph{reeling in (the snakes) $a$ and $b$}. Observe that, due to coverage of the edge $\{a,u\}$, and running intersection, there is a simple path of bags $p_{ua}$ starting
at $B_1$ that all contain $u$ such that the endpoint of the path also contains $a$. The path will necessarily be entirely on the left of the decomposition. Due to coverage of the edge $\{b,v\}$ there is an analogously-defined simple path $p_{vb}$. (Note that $p_{ua}$ and $p_{vb}$
both exit $B_1$ via the same bag $B'$. If they exited via different bags, coverage of the
edge $\{a,b\}$ would force at least one of $a,b$ to be in $B_1$, yielding a contradiction). Now, in the bags along $p_{ua}$, except $B_1$, we relabel $u$ to be $a$, and in the bags
along $p_{vb}$, except $B_1$, we relabel $v$ to be $b$. This is no longer necessarily a valid
tree decomposition, because coverage of the edges $\{u,a\}$ and $\{v,b\}$ is no longer
guaranteed, but we shall address this in due course. Next we delete the vertices
$u,w,v$ from all bags on the left of the decomposition, except $B_1$; they will not be needed. (The only reason that $w$ would be in a bag on the left, would be to meet $c$, since $B_1$ and $B_2$ already cover the edges $\{u,w\}$ and $\{w,x\}$. But then, due to coverage of the edge $\{c,d\}$ and the fact that $d$ only appears on the right of the decomposition, running intersection would put at least one of $c,d$ in $B_1$, contradiction.) Observe that $B'$ contains $\{a,b\}$. We replace $B_1$ with 5 copies of itself, and place these bags in a path such that the leftmost copy is adjacent to $B'$, the rightmost copy is adjacent to $B_2$, and all other bags that were originally adjacent to $B_1$ can (arbitrarily) be made adjacent to the leftmost copy. In the 5 copied bags we replace $\{u,w,v\}$ respectively with: $\{a,u',b\}$, $\{u',b,v'\}$, $\{u',u,v'\}$, $\{v',u,v\}$ and $\{u,w,v\}$. It can be verified that this
is a valid tree decomposition for $G'$, and our construction did not inflate the treewidth - we either deleted vertices from bags or relabelled vertices that were already in bags - so we are done. The operation can easily be telescoped, if desired, to achieve
an arbitrarily long ladder.

\smallskip 

\noindent
\textbf{Case 4. $P$ contains at least one bag other than $B_1$ and $B_2$.} 
\begin{observation}
\label{obs:obs_v_w_everywhere}
All bags in $P$ contain $v,w$, by the running intersection property.
% MM: shortened to one line
% All the bags in $P$ contain $v$ and $w$, by the running intersection property.
\end{observation}

We partition the bags of the decomposition into (i) $B_1$, (ii) bags \emph{left} of $B_1$, (iii) $B_2$, (iv) bags \emph{right} of $B_2$, (v) all other bags (which we call the \emph{interior}).

\begin{hidden}
\begin{itemize}
\item $B_1$,
\item bags \emph{left} of $B_1$,
\item $B_2$,
\item bags \emph{right} of $B_2$,
\item all other bags. These are exactly the inner bags of the path $P$ (i.e. all bags on $P$ excluding its endpoints $B_1$ and $B_2$), plus all bags that can be reached from an inner bag of $P$ without passing through $B_1$ or $B_2$. For this reason we call this group the \emph{interior}. 
\end{itemize}
\end{hidden}

%At this point we know the following. We have $b,d,x \not \in B_1$, $a,c,u \not \in B_2$, due to the assumption that Case 2 does not hold, and all the bags in $P$ contain both $v, w$ due to running intersection.
% MM: shortened to:
Recall that $b,d,x \not \in B_1$, $a,c,u \not \in B_2$ (because Case 2 does not hold). 
\begin{observation}
\label{obs:no_u_no_x_in_interior_or_wrong_side}
No bag in the interior contains $u$ or $x$. $B_1$ does not contain $x$, and no bag on the left contains $x$. Symmetrically, $B_2$ does not contain $u$, and no bag on the right contains $u$.
\end{observation}
\begin{proof}
Recall Observation \ref{obs:obs_v_w_everywhere}. If some bag in the interior contained $u$ or $x$, we could
due to running intersection (in particular: due to $u \in B_1, x \in B_2$) find two bags on $P$ containing $\{u,v,w\}$ and $\{v,w,x\}$ that were closer than $B_1$ and $B_2$, contradiction. Next, we have already established that $x \not \in B_1$ and $u \not \in B_2$. If $x$ was on the left, then running intersection would put $x \in B_1$ (because $x \in B_2$), contradiction. If $u$ was on the right, then running intersection would put $u \in B_2$ (because $u \in B_1$), contradiction.
\end{proof}

\begin{observation}
\label{obs:a_and_d_magnet}
At least one of the following is true: $a \in B_1$,  $a$ is in a bag on the left. Symmetrically, at least one of the following is true: $d \in B_2$, $d$ is in a bag on the right.
\end{observation}
\begin{proof}
%This is because 
The edge $\{a,u\}$ (respectively, the edge $\{d,x\}$) needs to be in a bag, %somewhere,
and from Observation \ref{obs:no_u_no_x_in_interior_or_wrong_side} $u$ and $x$ are restricted in their possible locations.
\end{proof}

Now, suppose $w$ is somewhere on the left. We will show that then either $w$ can be deleted from the bags on the
left, or Case 2 holds. A symmetrical analysis will hold if $v$ is somewhere on the right. Specifically, the only possible reason for $w$ to be on the left would be to cover the edge $\{w,c\}$ -- all other edges
incident to $w$ are already covered by $B_1$ and $B_2$. If no bags on the left contain $c$, we
can simply delete $w$ from all bags on the left. On the other hand, if some bag on the left
contains $c$, then $c \in B_1$, because: $d \not \in B_1$, the need to cover the edge $\{c,d\}$, the presence of $d$ on the other `side' of the decomposition (Observation \ref{obs:a_and_d_magnet}), and running intersection. So we have that $c,u,w,v \in B_1$. This bag already covers all edges
incident to $w$, except possibly the edge $\{w,x\}$. To address this, we replace $w$ everywhere
in the tree decomposition with $x$ - this is a legal tree decomposition because some bag contains $\{w,x\}$ - and then add a bag $B' = \{u,w,x,c\}$ pendant to $B$. This new bag serves to cover all edges incident to $w$.  But $B_1$ now contains $u,v,c,x$, so Case 2 applies, and we are done! Hence, we can assume that $w$ is nowhere on the left, and, symmetrically, that $v$ is nowhere on the right.
In fact, the above argument can, independently of $w$, be used to trigger Case 2 whenever $c \in B_1$ or $b \in B_2$.

So at this stage of the proof we know: $b,c,d,x \not \in B_1$ (and $c$ is not on the left) and $a,b,c,u \not \in B_2$ (and $b$ is not on the right).

\smallskip

%\noindent
\emph{Subcase 4.1:} Suppose $a \not \in B_1$. Then, $a$ must only be on the left. It cannot be in the interior (or on the right) because the edge $\{u,a\}$ must be covered, $a \not \in B_1$, $u \in B_1$, and $u$ is not in the interior (Observation \ref{obs:no_u_no_x_in_interior_or_wrong_side}). 
%
%Due to the fact that 
Because $a$ is on the left, and 
%the fact that 
because some bag must contain the edge $\{a,b\}$, $b$ must also be on the left. In fact $b$ is only on the left. The presence of $b$ both on the left and in the interior (or on the right) would force $b$ into $B_1$ by running intersection, contradicting the fact that $b \not \in B_1$. So $a,b$ are only on the left. We are now in a situation similar to Subcase 3.2. We use the same \emph{reeling in $a$ and $b$} construction and we are done.

\smallskip

%The edge $\{a,u\}$ means that there is a path of bags $p_{ua}$ 
%that begins in $B_1$ where all bags on the path contain $u$, and the last
%bag on the path contains $u$. (Note that this path, apart from $B_1$, is
%completely on the left of the decomposition). There is also a path $p_{vb}$, defined analogously. Observe that the first bag (after $B_1$) on $p_{ua}$ and $p_{vb}$ is the same, call this bag $B'$. (If they were different, the edge $\{a,b\}$ and running intersection would force at least one of $a$ and $b$ into $B_1$, contradiction.) We are now in a situation very similar to Subcase 3.2.
%\noindent
\emph{Subcase 4.2:} Suppose $a \in B_1$. Note that here $u$ has all its incident edges covered by $B_1$, so $u$ can be deleted from all other bags.

Recall that $b \not \in B_1$. Due to edge $\{b,v\}$ some bag must contain both $v$ and $b$. Suppose there is such a bag on the left. We attach a new bag $\{a,u,w,v\}$ pendant to $B_1$ and  delete $u$ from $B_1$. We put $x$ in $B_1$ and to ensure running intersection we replace $v$ with $x$ in all bags anywhere to the right of $B_1$. This is safe, because in the part of the decomposition right of $B_1$, $v$ only needs to meet $x$ (and not $b$, because $v$ meets $b$ on the left). Thus, $B_1$ now contains $\{a,v,w,x\}$ and Case 2 can be applied.

%If such a bag is on the left, then (by running intersection and the absence
%of $b$ from $B_1$) $b$ is not in the interior or on the right.

%Moreover, there is a simple path of bags $p_{vb}$ starting at $B$ which ends at a bag on the left containing $b$, whereby all the bags on the path contain $v$. Similar to earlier we relabel $v$ to $b$ along this path (except in $B_1$) and delete $v$ from all other bags on the left.  We replace $u$ with
%$b$  in $B_1$ and attach a new bag $\{a,u,w,v\}$ pendant to $B_1$. Note that the relabelling
%of the $p_{vb}$ path and the deletion of $v$ from other bags on the left is safe because there are no paths from $v$ to $x$ on the left (by Observation
%\ref{obs:no_u_no_x_in_interior_or_wrong_side}) -- so coverage of the edge $\{v,x\}$, and
%running intersection, is not damaged.
Hence, we conclude that $\{v,b\}$ is not in a bag on the left. Because of this $v$ can safely be deleted from all bags on the left. That is because any path $p_{vb}$ that starts at $B_1$ and finishes at a bag containing $b$ must go via the interior. 
In fact, such a path must avoid $B_2$, and is thus entirely contained in the interior. It avoids $B_2$ because $a,b \not \in B_2$ and $\{v,b\}$ cannot be in a bag to the right: if it was, coverage of edge $\{a,b\}$, the fact that $a \in B_1$ and running intersection would mean that at least one of $a$ and $b$ is in $B_2$, yielding a contradiction.

The only case that remains is $a \in B_1$, $\{v,b\}$ is not in a bag on the left and thus $p_{vb}$ is in the interior. 
By symmetry, we assume that $d \in B_2$, $\{w,c\}$ is not in a bag on the right and thus $p_{wc}$ is in the interior. 
Consider any path $p_{ab}$ starting at $B_1$, defined in the now familiar way. Note that no bag on the left of $B_1$ can contain $b$. This is because $\{v,b\}$ is in a bag in the interior: hence if $b$ was also on the left, $b$ would then by running intersection be in $B_1$ and we would be in an earlier case. This means that $p_{ab}$ must go via the interior.
Suppose the following operation gives a valid tree decomposition: delete $u$ from $B_1$, attach a new bag $B^{*} = \{a,u,w,v\}$ pendant to $B_1$, and relabel all
occurrences of $a$ along the path $p_{ab}$ (except in $B_1$) with $b$. Then we are done, because we are back in Case 2. A symmetrical situation holds for the path $p_{dc}$.

\begin{figure}[t]
\centering
\includegraphics[scale=1.2]{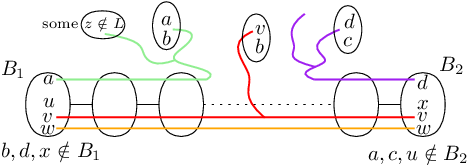}
\caption{%The most difficult subcase to deal with. Here,
Path $p_{ab}$ goes via the interior, but it cannot be relabelled to $b$ because it is used by other paths $p_{az}$ to some neighbour $z$ of $a$ that does \emph{not} lie on the ladder. %Symmetrically for $p_{dc}$.
}
\label{fig:trickycase} 
\end{figure}

Assume therefore that this transformation does not give a valid tree decomposition. This is the most complicated case to deal with. It is depicted in Fig. \ref{fig:trickycase}. The issue here is that the path $p_{ab}$ (respectively, $p_{dc}$) necessarily goes via the interior, but cannot be relabelled with $b$ (respectively, $c$) because the path is also part of $p_{az}$ (respectively, $p_{dz}$) where
$z$ is some non-ladder vertex that is adjacent to $a$ (respectively $d$). We deal with this as follows. We argue that some bag in the decomposition \emph{must} contain $a,b,v$ (and possibly other vertices). Suppose this is not the case. By standard chordalization arguments, every chordalization adds at least one diagonal edge to every square of the ladder. If $\{a,b,v\}$ are not together in a bag, then this is because the corresponding chordalization did not add the diagonal $\{a,v\}$ to square $\{a,u,b,v\}$. Hence,
the chordalization must have added the diagonal $\{u,b\}$. This would in turn mean that some bag contains $\{a,u,b\}$. Such a bag must be on the interior, because this is the only place that $b$ can be found. However, no bags in the interior contain $u$ -- contradiction. 

Hence, some bag $B'$ indeed contains $\{a,b,v\}$. Again, because $b$ is only on the interior, $B'$ must be in the interior. There could be multiple such bags, but this does not harm us. Let $B_{v\text{-done}}$ be the rightmost bag on the path $P$ that is part of a path, starting from $B_1$, from $a$ to some bag $B'$  containing $\{a,b,v\}$. Let $B_{a\text{-done}}$ be the rightmost bag on $P$ that contains $a$. Note that $B_{v\text{-done}}$ contains $a$ (because of
running intersection: $a \in B_1$ and $a \in B'$) and $v, w$ (because it lies on $P$). We also have $a,v,w \in B_{a\text{-done}}$. By construction, $B_{v\text{-done}}$ is either equal to $B_{a\text{-done}}$ or left of it. This is important because it means that the only reason $v$ might need to be in bags to the right of $B_{a\text{-done}}$ is to reach a bag containing $x$ (i.e. to cover the edge $\{v,x\}$) -- all other edges are already covered elsewhere in the decomposition; in particular, edge $\{v,b\}$ is covered by the bag $B'$ containing $\{a,b,v\}$. See Fig. \ref{fig:trickycase2} for clarification.
\begin{figure}[t]
\centering
\includegraphics[scale=1.2]{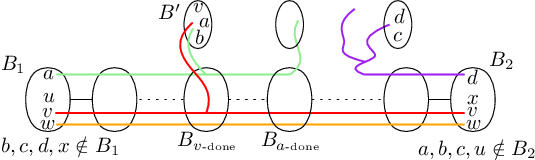}
\caption{The bags $B'$, $B_{v\text{-done}}$ and $B_{a\text{-done}}$ illustrated. Note that $B_{a\text{-done}}$ cannot be the penultimate bag on the path $P$ from $B_1$ to $B_2$, due to the presence of $d$ in that bag.}
\label{fig:trickycase2} 
\end{figure}
Recall that none of the bags on the path $P$ contain both $a$ and $d$. (If they did, there would be a bag containing $\{a,d,v,w\}$ and we would be in Case 2, done.) We also know that some path $p_{dc}$ goes via the interior and thus that the penultimate bag on $P$ (i.e. the one before $B_2$) thus definitely contains $d$. (To clarify: $c \not \in B_2$, $d \in B_2$, the edge $\{c,d\}$ must be covered, and $c$ is only in the
interior). Combining these insights tells us that this penultimate bag definitely does \emph{not} contain $a$, and hence $B_{a\text{-done}}$ is \emph{not} equal to the penultimate bag; it is further left. This fact is crucial. Consider $B_{a\text{-done}}$ and the bag immediately to its right on $P$. Between these two bags we insert a copy of $B_{a\text{-done}}$, call it $B_r$, remove $a$ from $B_r$ (i.e. forget it), and add the element $x$ to it instead. Finally, we switch
$v$ to $x$ in all bags on $P$ right of $B_r$, including $B_2$ itself, and delete $v$ from all bags in the tree decomposition that are anywhere to the right of $B_r$; there is no point having them there. It requires some careful checking but this is a valid (minimum-width) tree decomposition. Moreover, $B_r$ contains $w,v,x$. The fact that $B_{a\text{-done}}$ was not the penultimate bag of $P$, means that the length of the path from $B_1$ to $B_{r}$ is strictly less than the length of the path from $B_1$ to $B_2$: contradiction on the assumption that these were the closest bags containing $\{u,w,v\}$ and $\{w,v,x\}$ respectively. We are done.
% main end of proof
\end{proof}

We now deal with the situation when the $tw(G) \geq 4$ assumption is removed. 
%As a warm-up we show how
%a starting length of 5 is sufficient.
\begin{hidden}
\begin{corollary}
\label{cor:mainAllTreewidth}
Let $G$ be an undirected graph of arbitrary treewidth. There is a universal constant $C'$ such that if there exists a
ladder $L$ of length length $C'$ or longer in $G$, the ladder can be increased arbitrarily in length without
altering (in particular: increasing) the treewidth.
\end{corollary}
\begin{proof}
If $L$ is disconnecting, a ladder of length 1 is already sufficient (Lemma \ref{lem:disconnecting}). So assume that $L$ is not disconnecting. From Observation \ref{obs:k4} we know $tw(G) \geq 3$. Now, recall the function $f$ from Theorem \ref{thm:unbounded}. We take $C' = \max(f(3), 3)$, where $f(3)$ indicates a starting length from which ladders can be extended in graphs of treewidth 3 without increasing the treewidth. The second term in the max operator is the constant 3 obtained from the statement of Theorem \ref{thm:main} and which holds for all graphs $G$ such that $tw(G) \geq 4$.
\end{proof}

The above corollary is somewhat unsatisfactory because, although the constant $f(3)$ does exist, it is much larger than the constant 3 from Theorem \ref{thm:main}. However, by carefully analysing the case when $tw(G) \leq 3$ we can produce a more satisfactory result that again works for all treewidths.
\end{hidden}
\begin{lemma}
\label{lem:main2}
If $G$ has a ladder $L$ of length 5 or longer, the ladder can be increased in length arbitrarily without
altering (in particular: increasing) the treewidth. This holds irrespective of the treewidth of $G$.
\end{lemma}
\begin{proof}
Let $L$ be a ladder of length 5 or longer.
We can assume that $L$ is not disconnecting and $tw(G) \leq 3$. We select the three most central squares and label these as in Fig. \ref{fig:ladder}. These are flanked on both sides by at least one other square. Hence, $a,b,c,d$ each has exactly one neighbour  outside the 3 squares, let us call these $a', b', c',d'$ respectively, where $\{a',b'\}$ is an edge and $\{c',d'\}$ is an edge. Now, $tw(G) = 3$ because $L$ is not disconnecting. The only part of the proof of Theorem \ref{thm:main} that does not work for $tw(G)=3$ is Case 1 and (indirectly) Case 2 because these create size-5 bags. We show that neither case can hold.

Consider Case 1.  Let $B$ be a bag containing one of the three most central squares $S$ of the ladder (these are the only squares to which Case 1 is ever applied).
%We assume the square is $\{u,v,w,x\}$ but the argument we give will work for any of the three central squares}.
A \emph{small} tree decomposition is one where no bag is a subset of another. If a tree decomposition is not small, then by running intersection it must contain two adjacent bags $B^{\dagger}, B^{\ddagger}$ such that $B^{\dagger} \subseteq B^{\ddagger}$. The two bags can then be safely merged into $B^{\ddagger}$. By repeating this a small tree decomposition can be obtained without raising the width of the original minimum-width decomposition. Furthermore, some bag $B$ will still exist containing $S$. If $B$ has five or more vertices we immediately have $tw(G) \geq 4$ and we are done. Otherwise, let $B'$ be any bag adjacent to $B$; such a bag must exist because $G$ has more than 4 vertices. Due to the smallness of the decomposition we have $B \not \subseteq B'$ and $B' \not \subseteq B$. Hence, $B \cap B' \subset B$ and $B \cap B' \subset B'$. A \emph{separator} is a subset of vertices whose deletion disconnects the graph. Now, $B \cap B'$ is by construction, and the definition of tree decompositions a separator of $G$. However, due to our use of the three central squares, $S$ is not a separator, and no subset of it is a separator either;
%\footnote{The same actually holds for $\{a,u,b,v\}$ and $\{w,c,x,d\}$ but we do not use this fact.}
the inclusion of $a', b', c', d'$ and the edges $\{a',b'\}$ and $\{c',d\}$, alongside the fact that $L$ is not disconnecting, ensure this. This yields a contradiction. Hence Case 1 implies $tw(G) \geq 4$ i.e. it cannot happen when $tw(G)=3$.

We are left with Case 2. This case replaces the three centremost squares with a single square, and deletes any diagonals that this single
square might have, to obtain a new graph $G'$. We have $tw(G') \leq tw(G)$, by minors. Note that $tw(G') \geq 3$ because the shorter ladder in $G'$ (which has length at least 3) is still disconnecting. Hence, $tw(G') = tw(G) = 3$. The decomposition $\mathbb{T}'$ of $G'$ obtained by projecting the contraction operations onto the tree decomposition $\mathbb{T}$ of $G$, is a valid tree decomposition (as argued in Case 2) with
the property that the width of $\mathbb{T}'$ is less than or equal to the width of $\mathbb{T}$. $\mathbb{T}'$ cannot have width less than 3, so it must
have width 3. Hence it is a tree decomposition of $G'$ in which all bags have at most four vertices. We then transform $\mathbb{T}'$ into a small tree composition: this does not raise the width of the decomposition, and every bag prior to the transformation either survives or is absorbed into another. Consider
the bag $B'$ containing $H_1, H_2, L_1, L_2$. The presence of $a', b', c,', d'$ in $G'$ and the fact that the ladder in $G'$ is not disconnecting, means that $H_1, H_2, L_1, L_2$ is not a separator for $G'$, and neither is any subset of those four vertices. But the intersection of $B'$ with any neighbouring bag \emph{must} be a separator. Hence $B'$ must contain a fifth vertex, contradiction. So Case 2 cannot happen when $tw(G)=3$.
\end{proof}

We can, however, still do better. Consider first the following auxiliary lemma.

\begin{lemma}
\label{lem:pointed}
If $G$ has $tw(G) \geq 3$ and a ladder $L$ of length 1 or longer whereby at least one of the four cornerpoints of the ladder has degree 2, the ladder can be increased in length arbitrarily without
altering (in particular: increasing) the treewidth. 
\end{lemma}
\begin{proof}
Let $\{a,b,c,d\}$ be the four cornerpoints of the ladder and let $c$ be a degree-2 cornerpoint. Assume as usual that $c$ and $d$ are part of square $\{w,x,c,d\}$. If we suppress
$c$ and relabel vertex $d$ as $cd$ we create a triangle
$\{w,x,cd\}$ in $G$ without altering its treewidth. Take any minimum-width decomposition. This triangle must be contained in some bag $B$ of the decomposition. Pendant to $B$ we attach a new chain of bags $\{w,x,w',cd\}$, $\{x,w',x',cd\}$. This is a valid tree decomposition for the graph obtained from $G$ by inserting a new rung in the ladder $\{w',x'\}$ parallel to edge $\{w,x\}$. The construction can be iterated if desired to insert more rungs in the ladder, by attaching bags pendant to bag $\{x,w',x',cd\}$. Once completed the degree 2 vertex can be re-introduced via subdivision, if desired.
\end{proof}\\
\\
\begin{theorem}
\label{thm:main3}
If $G$ has a ladder $L$ of length 4 or longer, the ladder can be increased in length arbitrarily without
altering (in particular: increasing) the treewidth. This holds irrespective of the treewidth of $G$.
\end{theorem}
\begin{proof}
As usual, if $tw(G) \geq 4$ we can use Theorem \ref{thm:main},
and if $L$ is disconnecting then we are done thanks to Lemma \ref{lem:disconnecting}.
So, let $G$ be a graph with ladder $L$ with four squares that is not disconnecting. From Observation \ref{obs:k4} we have $tw(G)=3$. 
Let $a',a,u, w,c$ be the vertices on the top of the ladder and $b', b,v,x,d $ be the vertices on the bottom. Our goal as usual is to show that adding a square does not increase the treewidth of $G$.
We henceforth assume that $G$ is (vertex) biconnected. This is because the treewidth of a graph is the maximum treewidth ranging over all biconnected components of the graph. The ladder, both before and after lengthening, belongs to a single biconnected component. So we henceforth focus only on that
component.

 As in the proof of Theorem \ref{thm:main} we focus on the three squares defined by vertices $a,u,w,c$ and $b,v,x,d$. We have an extra square on the left side - $a', a, b', b$ - and this has the same `buffer' role as in the proof of Lemma \ref{lem:main2}. However, there is no extra buffer square on the right of the ladder, and this causes some mild complications.

%We will step through the proof of Theorem \ref{thm:main}, pointing out where the differences lie. Before this,
We begin with several observations. Whenever, in the proof of Theorem \ref{thm:main}, Case 1 or Case 2 are shown to apply to four vertices from $a,u,w, b,v,x$ (i.a. avoiding $c$ and $d$) then we are already done\footnote{When applying Case 2 in such a context, we leave $c$ and $d$ alone i.e. we only contract two squares of the ladder, not three.}. That is because, using exactly the same argument as in Lemma \ref{lem:main2}, those four vertices or a subset thereof cannot (after contraction, when Case 2 applies) be a separator, so there must be an extra vertex in the bag: $tw(G) > 3$, and a contradiction on the assumption $tw(G)=3$ is obtained.  The absence of a separator is due to the vertices $a', b'$, the edges $\{a',b'\}, \{a',a\}, \{b',b\}$ and the fact that we avoid vertices $c, d$, allowing them to assume the same role as $c', d'$ in the proof of Lemma \ref{lem:main2}. Hence, the main headache is when Case 1 or Case 2 is applied to four vertices involving $c$ and/or $d$. The problem is that, due to not having any knowledge about the non-ladder neighbours of $c$ and $d$, we cannot guarantee that the four vertices (or a subset thereof) do not form a separator. Hence, it is not possible to directly derive a contradiction on $tw(G)=3$. In such situations the size-5 bags introduced by Case 1 and Case 2 might, therefore, inflate the treewidth. However, it is possible to circumvent this, as we shall see.

%Observe that if the ladder is disconnecting, we are done thanks to Lemma \ref{lem:disconnecting}. So henceforth we assume that $G$ is biconnected, $tw(G)=3$, and that the ladder is not disconnecting.

The fact that $G$ is biconnected and $L$ is not disconnecting means that at least one of the following holds:
\begin{itemize}
    \item There is a simple path that starts at $c$, avoids all other vertices on the ladder (in particular: $d$), and ends at $a'$ or $b'$;
    \item There is a simple path that starts at $d$, avoids all other vertices on the ladder (in particular: $c$), and ends at $a'$ or $b'$;
\end{itemize}
(Note that it is permitted that these paths intersect, perhaps multiple times, at vertices distinct from $c$ and $d$).
Now, suppose \emph{both} these paths exist.  In this situation we are done, because Cases 1 and Case 2 can be applied in their unconstrained form i.e. they do not even need to avoid $c$ and $d$. This is because, after contracting the three squares to one, the single square remaining (or a subset thereof) cannot be a separator. This is because $c$ and $d$ are independently of each other connected to the other side of the ladder (and as observed above the square $a', a', b, b'$ basically has the same separator-preventing function at the other end of the ladder). Hence, the proof of Theorem 2 goes through essentially unchanged, the only difference being that Case 1 and Case 2 now generate contradictions on the assumption $tw(G)=3$. 

Hence, we assume that only one such path exists. For now, suppose this path starts at $d$. Due to biconnectivity, and the absence of the second path, there are exactly two possibilities:
\begin{enumerate}
    \item $c$ has degree 2, in which case its only neighbours are $d$ and $w$.
    \item The edge $\{c,d\}$ is a separator.
\end{enumerate}

If $\{c,d\}$ is a separator, then deleting this edge splits $G$ into $G_1$ and $G_2$, where $G_1$ is the connected component containing the ladder.  The treewidth of $G$ is equal to
the maximum of $tw(G_1 \cup \{c,d\})$ and $tw(G_2 \cup \{c,d\})$. This is because $\{c,d\}$ is a clique separator. In
particular, any tree decomposition of $G_1 \cup \{c,d\}$ (respectively, $G_2 \cup \{c,d\}$) must have $\{c,d\}$
 together in some bag, so any two such tree decompositions can be linked
 together via a single extra bag containing $\{c,d\}$ in order to obtain a tree decomposition for $G$. Now, $G_2 \cup \{c,d\}$ (and thus its treewidth) is unchanged if the ladder is extended, so we can focus on the graph $G_1 \cup \{c,d\}$. This brings us back into the situation that $c$ has degree 2. (An exactly symmetrical argument holds if the path had started at $c$.)
 
 Hence, at this point we can assume without loss of generality
 that exactly one of $c$ and $d$ has degree 2. We can invoke Lemma \ref{lem:pointed} and we are done.
\end{proof}
\sknov{\section{Tightness}}
\label{sec:tightness}
\noindent
The constant 4 in the statement of Theorem \ref{thm:main3} is equal to the constant obtained for the `bottleneck' case $tw(G)=3$. An improved constant 3 for this case is not possible, as Fig. \ref{fig:prism} shows. 

\begin{figure}
\centering
\includegraphics[scale=0.8,page=1]{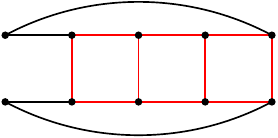}
\caption{A graph of treewidth 3 that contains a ladder with 3 squares, shown in red. Increasing the length of the ladder by 1 square increases the treewidth to 4.
% because it then has the pentagonal prism, one of the forbidden minors for graphs as treewidth at most 3, as a minor.
}
\label{fig:prism} 
\end{figure}

\noindent
From Theorem \ref{thm:main} we know that if $tw(G) \geq 4$ we can start from ladders of length~3. It is therefore natural to ask whether, 
if the treewidth of $G$ is sufficiently high, we can start from ladders of length 2 rather than 3. 
This is not possible, as we now show.

\begin{lemma}
\label{lem:alwaystight}
For every $t \geq 3$ there exists a graph $G_2(t)$ that contains a ladder of length 2, has treewidth exactly $t$ and such that if we increase the length of the ladder by one, 
creating $G_3(t)$, the treewidth increases to exactly $t+1$.
\end{lemma}

\begin{proof}   
We will prove this by showing that $G_2(t)$ \steven{has}
%have
treewidth at most $t$ and that $G_3(t)$
\steven{has}
%will have
treewidth at least $t+1$. \steven{When} combined with Lemma 3 \steven{this} gives us the desired result. 
   
We construct $G_2(t)$ as follows. \steven{First, we introduce vertices $1,2,3,4,5,6$ and the following edges which construct the (maximal) ladder of length 2:
\[
    \{ \{1,2\}, \{2,3\}, \{4,5\}, \{5,6\}, \{1,4\}, \{2,5\}, \{3,6\} \}\text{.}
\]}
%    Let vertices labelled with 1,2,3,4,5,6 be the ladder of size 2. Starting with 1 and going down with 2 and 3, and then on the other side start with 4 then 5 and 6.
In addition to the ladder we have a clique containing $(t-1)$ vertices $7,8, \ldots, 7+t-2$.
%
%labeled by 7 up to and including $7+t-2$. 
%
%\steven{The ladder is not disconnecting}. Observe that Lemma 3 states that each connected graph that contains a \steven{non-disconnecting} ladder $L$ of length 2 has treewidth at least 3. So we may assume that $t\geq3$,
\steven{Given that $t \geq 3$ we thus always have a non-trivial clique that will contain at least 2 vertices, i.e. vertices 7 and 8.} 
The ladder is further connected to the clique as shown in Fig.~\ref{fig:tightstuff}, and described as follows: 
\begin{itemize}
    \item Vertices 1 and 3 are connected to every vertex of the clique, and
    \item Vertices 4 and 6 are connected to every vertex of the clique except vertex 7. 
\end{itemize}
%the figure below showing the general construction.
%\steven{Note that the ladder in the graph is not disconnecting, so by Lemma 3 the treewidth of the graph is at least 3}.

\begin{figure}[t]
\begin{center}
\begin{tikzpicture}[scale = 0.8]
    
    \node (G) at (-3,0) {$G_2(t)$};
    
    %lader
    \vertex [label=above:$1$] (1) at (0,1) {};
    \vertex [label=right:$2$] (2) at (0,0) {};
    \vertex [label=below:$3$] (3) at (0,-1) {};
    
    \vertex [label=left:$4$] (4) at (-1,1) {};
    \vertex [label=left:$5$] (5) at (-1,0) {};
    \vertex [label=left:$6$] (6) at (-1,-1) {};

    % Clique
    \vertex [label=below:$7$] (7) at (2,0) {};
    \vertex [label=below right:$8$] (8) at (3,0) {};
    \vertex [label=right:$7+(t-2)$] (t) at (5,0) {};
    
    \draw [line width = 1pt]
    (1) edge (3)
    (1) edge (4)
    (2) edge (5)
    (3) edge (6)
    (4) edge (6)
    (7) edge (1)
    (7) edge (3)
    (7) edge (8)
    (8) edge[dotted] (t)
    (1) edge[bend left=70] (8)
    (1) edge[bend left=70] (t)
    (4) edge[bend left=70] (8)
    (4) edge[bend left=70] (t)
    (3) edge[bend right=70] (8)
    (3) edge[bend right=70] (t)
    (6) edge[bend right=70] (8)
    (6) edge[bend right=70] (t);

    \draw [thick, decorate, decoration={brace, mirror, amplitude=5pt}] (t) -- (8) node [black, midway, yshift=10pt] {$C$};
    
\end{tikzpicture}
\caption{Graph $G_2(t)$ for $t \geq 3$. Vertices $7,8,\ldots,7+(t-2)$ form a clique; the edges of the clique are not depicted for simplicity. The graph has a ladder of length 2 and treewidth exactly~$t$.}
\label{fig:tightstuff} 
\end{center}
\end{figure}
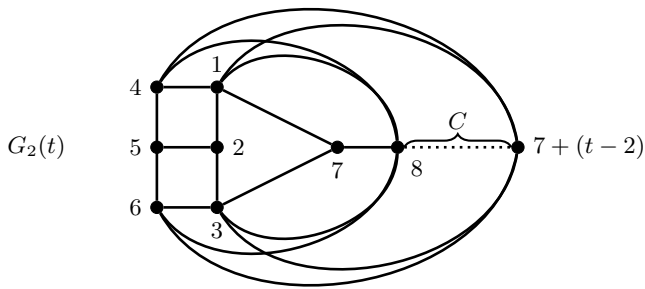
%   
%\noindent
Now we will \steven{prove} that $G_2(t)$
%will
\steven{has} treewidth at most $t$. We do this by giving a tree decomposition that has no bag with more than $t+1$ vertices. We define clique $C$ as the set of vertices in the aforementioned clique except vertex $7$, i.e. $\{8, \ldots, 7+t-2\}$. Note that $|C| = (7+t-2)-8+1 = t-2$. The tree decomposition has the following bags and edges:
\begin{enumerate}
    \item[bag 1:] $\{1,3,7\} \cup C $,
    \item[bag 2:] $\{1,3,5\} \cup C $ connected to bag 1,
    \item[bag 3:] $\{1,4,5\} \cup C $ connected to bag 2,
    \item[bag 4:] $\{3,5,6\} \cup C $ connected to bag 2,
    \item[bag 5:] $\{1,2,3,5\}$ connected to bag 2.
\end{enumerate}

Now we can simply check that these 5 bags form a valid tree decomposition. We check each of the three properties of a tree decomposition in turn (see the preliminaries).
\begin{description}
    \item[Property (tw1).] %$\cup_{i=1}^q B_i = V(G_2(t))$:\\
    The union of bags 1, 3, 4 and 5 is equal to the entire vertex set of $G_2(t)$, so clearly every vertex of $G_2(t)$ is in at least one bag.
    \item[Property (tw2).] %$\forall e = \{ u,v \} \in E(G_2(t)), \exists B_i \in \mathcal{B} \mbox{ s.t. } \{u,v\} \subseteq B_i$;\\
    Every edge in the clique \steven{is} covered by bag 1. The edges of the ladder are covered by bags 3, 4 and 5. The remaining edges connecting the ladder to the clique are covered by bags 1, 3
    %4
    and
    %5
    \steven{4}. So this property is also satisfied. 
    \item[Property (tw3).] %$\forall v \in V(G_2(t))$, all the bags $B_i$ that contain $v$ form a connected subtree of $\mathbb{T}$:\\
    Clearly satisfied.
\end{description}
Every \steven{bag}
%bad
contains \steven{at most} $t+1$ elements so the treewidth of $G_2(t)$ is less than or equal to $t$.

Now we move to $G_3(t)$ and show that it has treewidth \steven{at least} $t+1$. Let vertex $l$ and $r$ be the two vertices that we add to increase the length of the ladder, see Fig \ref{fig:g3}. 
We prove this lower bound by constructing a bramble with a \sknov{minimum} hitting set that has $t+2$ elements. From the result of Seymour and Thomas \cite{seymour1993graph} this means that $G_3(t)$ has treewidth at least $t+1$.\\

\begin{figure}[t]
\begin{center}
\begin{tikzpicture}[scale = 0.8]
    
    \node (G) at (-3,-0.5) {$G_3(t)$};
    
    %lader
    \vertex [label=right:$1$] (1) at (0,1) {};
    \vertex [label=right:$2$] (2) at (0,0) {};
    \vertex [label=right:$3$] (3) at (0,-2) {};
    
    \vertex [label=left:$4$] (4) at (-1,1) {};
    \vertex [label=left:$5$] (5) at (-1,0) {};
    \vertex [label=left:$6$] (6) at (-1,-2) {};

    % Clique
    \vertex [label=below:$7$] (7) at (2,-0.5) {};
    \vertex [label=below right:$8$] (8) at (3,-0.5) {};
    \vertex [label=right:$7+(t-2)$] (t) at (5,-0.5) {};

    \vertex [label={[text=red]left:$l$}, red] (l) at (-1,-1) {};
    \vertex [label={[text=red]right:$r$}, red] (r) at (0,-1) {};
    
    \draw [line width = 1pt]
    (1) edge (3)
    (1) edge (4)
    (2) edge (5)
    (3) edge (6)
    (4) edge (6)
    (7) edge (1)
    (7) edge (3)
    (7) edge (8)
    (l) edge[dashed, red] (r)
    (8) edge[dotted] (t)
    (1) edge[bend left=70] (8)
    (1) edge[bend left=70] (t)
    (4) edge[bend left=70] (8)
    (4) edge[bend left=70] (t)
    (3) edge[bend right=70] (8)
    (3) edge[bend right=70] (t)
    (6) edge[bend right=80] (8)
    (6) edge[bend right=80] (t);

    \draw [thick, decorate, decoration={brace, mirror, amplitude=5pt}] (t) -- (8) node [black, midway, yshift=10pt] {$C$};
    
\end{tikzpicture}
\caption{The graph $G_3(t)$, $t \geq 3$, constructed by increasing the length of the ladder in $G_2(t)$ by one. The new edge is $\{l,r\}$. A bramble argument is used to show the treewidth of $G_3(t)$ is a least $t+1$.}
\label{fig:g3}
\end{center}
\end{figure}
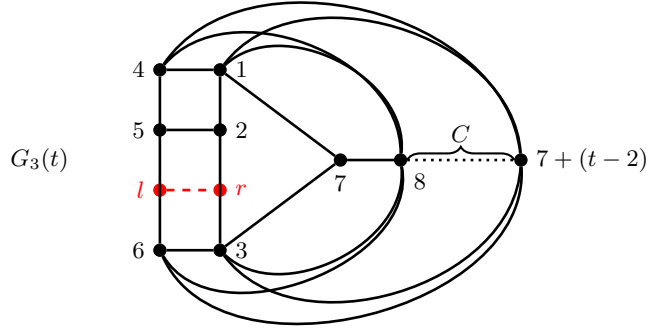

A \emph{bramble} of $G_3(t)$ is a set of connected subgraphs of $G_3(t)$ such that each subgraph in this set touches all others. Two subgraphs \emph{touch} each other if they have a vertex in common or there is an edge in the graph that connects the two subgraphs.\\

We construct the \steven{following} set of $|C|+5$ subgraphs $B$, show that $B$ is indeed a bramble, and then describe the size of a minimum hitting set of $B$. We describe each subgraph by a subset of vertices, understanding that we are taking the subgraph induced by this subset:
\begin{itemize}
    \item Every vertex from clique $C$ is a singleton subgraph
    \item Vertex 1 is a singleton subgraph
    \item $\{4,5\}$
    \item $\{2,r,l,6\}$
    \item $\{2,r,3,7\}$
    \item $\{l,6,3,7\}$
\end{itemize}

First we verify that $B$ is a valid bramble. It is easy to check that each subgraph is connected. We now show that the subgraphs are pairwise-touching:
\begin{itemize}
   \item Every singleton in $C$ \steven{is} connected to all other singletons (including $\{1\}$) and \steven{is} connected to vertices 4,6 and 7, therefore touching each other set.
   \item $\{1\}$ touches $\{4,5\}$ via edge $\{4,1\}$, touches $\{2,r,l,6\}$ via edge $\{1,2\}$ and touches $\{2,r,3,7\}$ and $\{l,6,3,7\}$ via edge $\{1,7\}$.
   \item $\{4,5\}$ touches the non-singleton sets via the edges $\{5,2\}$ and $\{5,l\}$.
   \item $\{2,r,l,6\}$ touches $\{2,r,3,7\}$ and $\{l,6,3,7\}$ by sharing vertices 2 and 6.
   \item $\{2,r,3,7\}$ touches $\{l,6,3,7\}$ by sharing vertex 7.
   \item $\{l,6,3,7\}$ using the same arguments as above.
\end{itemize}

So $B$ is indeed a bramble. Now we construct a \sknov{minimum} hitting set $H$. Clearly every singleton of $B$ must be in $H$. There are $t-1$ singletons in $B$, namely the $t-2$ constructed from $C$ and $\{1\}$. They do not share any of their vertices with another subgraph in the bramble. The same is true for $\{4,5\}$ so we add an arbitrary vertex from $\{4,5\}$, say vertex 4, to $H$. The remaining sets $\{2,r,l,6\}$, $\{2,r,3,7\}$ and $\{l,6,3,7\}$ can be hit by vertices 6 and 7, which is optimal given that they do not have a single element in common. $H$ has $t+2$ elements. So we conclude that the treewidth of $G_3(t)$ is at least $t+1$ and we are done.

\end{proof}

\medskip

\section{Resolving an open problem from phylogenetics}

\sknov{The research in this article was originally inspired by a question arising in phylogenetics, a subfield of bioinformatics. Namely: does the common chain reduction rule on two unrooted, binary
phylogenetic trees preserve the treewidth of the display graph? In this section we use the results from this article to answer this
affirmatively, and show that due to the restricted structure of display graphs slightly stronger bounds can be obtained than on general graphs. We start with some background and definitions.}

\subsection{The subtree and chain reduction rules are treewidth-preserving in the display graph}
 An (unrooted, binary) phylogenetic tree on a set of discrete labels $X$
 representing a set of species, is an undirected, connected, binary tree whose leaves are bijectively labelled by $X$. Due to this bijection we often refer to leaves and labels interchangeably. Two phylogenetic trees $T_1, T_2$, both on $X$, are defined to be equal if there is an isomorphism from one to the other that preserves the labels $X$. 

The \emph{display graph}
$D=D(T_1, T_2)$ of two unrooted binary phylogenetic trees $T_1, T_2$ on $X$ is obtained by identifying leaves with the same label. Display graphs have been quite intensively studied in recent years, see e.g. \cite{bryant2006compatibility,kelk2016monadic,fernandez2018compatibility,janssen2018treewidth,van2022embedding}. We note that the assumption $T_1 \neq T_2$ guarantees that $|X| \geq 4$ and that the display graph contains a $K_4$ minor, and thus has treewidth at least 3. For convenience we thus henceforth assume that $T_1 \neq T_2$. This is a very reasonable assumption because it is easy to check in polynomial time whether two phylogenetic trees are equal (equivalently, that the display graph has treewidth at most 2). The fact that the display graph has treewidth at least 3 is
useful because it allows us to suppress degree-2 nodes in the display graph without altering the treewidth or worrying about the whole display graph vanishing. In particular, we can suppress the degree-2 nodes that are created in the formation of the display graph when vertices with the same leaf label are identified.

The \emph{subtree reduction} is a data reduction rule very often used to simplify phylogenetic trees when computing a dissimilarity measure between them. Let $x,y$ be distinct labels in $X$. If $x,y$ have a common parent in $T_1$ and a common parent in $T_2$, then the \emph{cherry reduction} deletes the leaves $x$ and $y$ from both trees and assigns label $xy$ to the parent. The subtree reduction is simply when the cherry reduction is applied to exhaustion.

It was shown in \cite{kelk2017treewidth} that if one applies the subtree reduction rule to $T_1, T_2$ to
obtain new trees $T'_1, T'_2$ then $tw(D(T'_1, T'_2)) = tw(D(T_1, T_2))$. The question arose whether another frequently encountered data reduction rule,  the \emph{common chain reduction} rule, is also treewidth-preserving in the display graph. The definition of a common chain is rather technical\footnote{For $n\geq 2$, let $C = (\ell_1,\ell_2\ldots,\ell_n)$ be a sequence of distinct taxa in $X$.  
We call $C$ an $n$-chain of $T$ if there exists a walk $p_{\ell_1},p_{\ell_2},\ldots,p_{\ell_n}$ in $T$ and the elements in $p_{\ell_2},p_{\ell_3},\ldots,p_{\ell_{n-1}}$ are all pairwise distinct. Note that $\ell_1$ and $\ell_2$ may have a common parent or $\ell_{n-1}$ and $\ell_n$ may have a common parent. Furthermore, if  $p_{\ell_1} = p_{\ell_2}$ or $p_{\ell_{n-1}} = p_{\ell_n}$ holds, then $C$ is said to be {\it pendant} in $T$. If a chain $C$ exists in both phylogenetic trees $T_1$ and $T_2$ on $X$, we say that $C$ is a {\it common chain} of $T_1$ and $T_2$.}, but in essence it is an uninterrupted sequence of leaves that exist in the same order in both trees. The main nuance is that the first two leaves in the sequence, and the last two leaves in the sequence, might be unordered in one or both trees. The common chain reduction rule simply reduces common chains to length $k$, for some given constant $k$. It is well known that reduction to length 3 preserves a number of commonly encountered phylogenetic dissimilarity measures. Is there a constant $k$ such that the common chain reduction rule preserves the treewidth of the display graph?

Theorem \ref{thm:main3} shows that such a $k$ definitely exists. Specifically, common chains with $k$ leaves induce ladders
in the display graph with $k-1$ squares. Hence, if we reduce common chains to length 5, we know that the corresponding ladder in the display graph is reduced to length 4, and thus (via Theorem \ref{thm:main3}) that the treewidth is preserved. The result can be summarized as follows:

\begin{lemma}
\label{lem:preserve1}
Let $T_1, T_2$ be two unrooted binary phylogenetic trees on the same set of taxa $X$, where $|X| \geq 4$ and $T_1 \neq T_2$. Then exhaustive application of the subtree reduction and the common chain reduction (where common chains are reduced to 5 leaf labels) does not alter the treewidth of the display graph.
\end{lemma}

The subtree and chain reductions are the centrepiece of many kernelization results in phylogenetics \cite{bulteau2019parameterized}. Now we have established that
these two reduction rules \emph{also} preserve treewidth in the display graph. We note that, when trying to compute phylogenetic distance measures or parameters by exploiting low treewidth in the display graph, this treewidth-preserving result does not help: it is actually more advantageous if the treewidth decreases. Yet, if we are using the treewidth of the display graph as a proxy for phylogenetic dissimilarity, as proposed in \cite{kelk2017treewidth}, these results show that these two reduction rules are safe.

\subsection{One step further: leveraging the restricted structure of display graphs}

\sknov{Display graphs are a restricted subclass of graphs},
so it is natural to ask whether it is treewidth-preserving to reduce common chains to 4 or perhaps even fewer labels (rather than the 5 labels stated in  Lemma \ref{lem:preserve1}). We note, by leveraging an example from \cite{kelk2017treewidth}, that truncation to 3 leaf labels (inducing the shortening of ladders to 2 squares in the display graph) is \emph{not} treewidth preserving.
If we take the display graph of the two phylogenetic trees shown in Fig. \ref{fig:chains} (far left), which have a common chain of length 3 on the leaf labels $\{a,b,c\}$, we get a ladder in the display graph with 2 squares. The display graph has treewidth 3. However, if we take the display graph of the two trees shown in Fig. \ref{fig:chains} (second from right), where the chain has been increased to length 4 (and the ladder thus to 3 squares), the treewidth of the display graph increases to 4.

The question thus arises whether truncation to 4 leaf labels is safe. It turns out that it is! With a little more effort we obtain
the following theorem.
\begin{theorem}
\label{thm:preserve2}
Let $T_1, T_2$ be two unrooted binary phylogenetic trees on the same set of taxa $X$, where $|X| \geq 4$ and $T_1 \neq T_2$. Then exhaustive application of the subtree reduction and the common chain reduction (where common chains are reduced to 4 leaf labels) does not alter the treewidth of the display graph. This is best possible, because
there exist tree pairs where truncation of common chains to length 3 does reduce the treewidth of the display graph (see Fig \ref{fig:chains}).
\end{theorem}
\begin{proof}
We prove this by showing that if $T_1$ and $T_2$ have a common chain $C$ with 4 leaves $(a,b,c,d)$, and that this chain is then made longer to obtain new trees $T'_1$ and $T'_2$, we have $tw(D) = tw(D')$ where $D=D(T_1,T_2)$ and $D'=D(T'_1,T'_2)$. Note firstly that the chain $C$ induces a ladder $L$ with 3 squares in $D$\footnote{To remain consistent with the formal definition of a ladder - in particular to ensure that it has four cornerpoints - it might be necessary to leave some degree 2 vertices in the display graph unsuppressed, or even to introduce degree 2 vertices via subdivision.  However, as discussed in the preliminaries this will not alter the treewidth.}. Hence, if $tw(D) \geq 4$ we can use Theorem \ref{thm:main} and we are done. Similarly, if $L$ is disconnecting then we are done via Lemma \ref{lem:disconnecting}. Hence, we can assume that $L$ is not disconnecting, and $tw(D)<4$, so $tw(D)=3$. 

Now, we argue that there must exist a leaf label $x \not \in \{a,b,c,d\}$ such that in one of the two trees, say $T_1$, $x$ is on the ``$a$'' side of the chain in the tree, and in the other $T_2$ on the  ``$d$'' side of the tree. If this was not so then in $D$ there would be no path from the left of the ladder to the right that does not pass through the chain i.e. $L$ is disconnecting, contradiction. Now, suppose that $C$ is pendant in $T_1$ and/or $T_2$; recall that a chain is \emph{pendant} if two of its outermost leaves have a common parent (and this occurs if and only if the tree has no other leaf labels on that side of the chain). This induces at least one degree-2 cornerpoint in $L$, so by Lemma \ref{lem:pointed} we are done. Hence, assume that $C$ is pendant in neither tree. This means $|X| \geq 6$ because each tree needs at least one leaf label on both sides of the chain. Observe that if there is a leaf label $y \not \in \{a,b,c,d,x\}$ such that in $T_1$ $y$ is on the $d$ side of the chain and in $T_2$ on the $a$ side of the chain, then $tw(D) \geq 4$, yielding a contradiction. To see that it has $tw(D) \geq 4$ observe that the graph shown in Fig. \ref{fig:chains} (far right) will be a minor of it. So, let $y \not \in \{a,b,c,d,x\}$ be any taxon that
is on the $d$ side of the chain in $T_1$; such a taxon must exist by virtue of the assumption that the chain is pendant in neither tree. Given that $y$ cannot be on the $a$ side of the chain in $T_2$, it must - like $x$ - thus be on the $d$ side of the chain in $T_2$. But then we are in the situation as shown in Fig. \ref{fig:buffercycle}. The fact that the right end of the ladder is part of the yellow-highlighted cycle, has exactly the same separator-blocking effect as the buffer square $\{a',b',a,b\}$ used in Theorem \ref{thm:main3}. Hence, the same proof can be used as that theorem, and we are done.
\end{proof}

\begin{figure}
\centering
\includegraphics{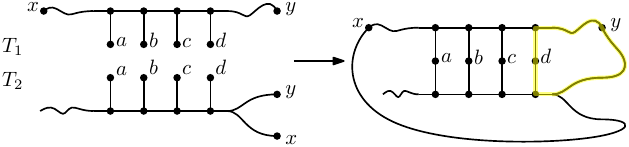}
\caption{If $T_1$ and $T_2$ have the shown structure, the 3 squares of the ladder induced in the display graph are flanked on the right by the yellow-highlighted cycle. This functions like the separator-blocking extra square in the proof of Theorem \ref{thm:main3}. This is one of the reason why chain-induced ladders can be safely reduced to 3 squares in display graphs (Theorem \ref{thm:preserve2}), rather than 4 as in general graphs.}
\label{fig:buffercycle} 
\end{figure}

\begin{figure}
\centering
\includegraphics[scale=0.8]{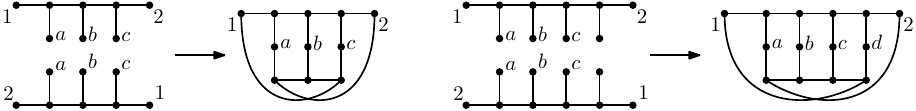}
\caption{Lengthening the common chain $\{a,b,c\}$ to $\{a,b,c,d\}$ in these phylogenetic trees causes the treewidth of the display graph to increase. Equivalently: shortening common chains to 3 leaf labels is not guaranteed to preserve treewidth in the display graph.}
\label{fig:chains} 
\end{figure}

\section{Discussion and future work}

\sknov{First, we note that the results in this article can be made constructive, in the following sense. 
Given an \emph{arbitrary} minimum-width tree decomposition for a graph with a short ladder, such that we know from the results of this article that the ladder can be safely extended without increasing the treewidth, we can transform it in polynomial time into a minimum-width tree decomposition for the graph after the ladder has been lengthened. 
This is because our proofs are based on many steps of the form ``inspect the tree decomposition and if necessary modify it in some simple way''. 
The only part of the proof that requires some non-trivial adjustment is the assumption, used in the proof of Theorem \ref{thm:main}, that we have access to a \emph{distance-minimizing} tree decomposition i.e. one that minimizes the number of bags between $B_1$ and $B_2$. 
This assumption is used to trigger a contradiction at several places within the proof: by showing \sknov{how to construct} a tree decomposition with smaller distance.
Fortunately, this can be made constructive by restarting the proof with the new tree decomposition that has smaller distance. 
Eventually this process stops and the proof completes without triggering the contradiction, i.e. constructively.}

\sknov{Second, we remark that our proof technique applies to structures slightly more general than ladders. 
For example, for ladders in which every square has at most one chord added. 
To see why, note that Case 1 of Theorem \ref{thm:main} concerns the situation when the tree decomposition induces a 4-clique in one of the squares of the ladder, and Case 2 reduces down to Case 1. 
In such a case, we have enough freedom to arbitrarily insert squares into the ladder with 0 or 1 chords in each square, and 2 if the starting treewidth is high enough. 
All other cases of the theorem concern ``reeling in the snakes'' such that a vertical rung of the ladder is in one bag and another vertical rung is in an adjacent bag. 
At this point a new path of bags is inserted between these two bags, allowing new vertical rungs to be inserted into the graph. 
This path of bags necessarily induces one chord per square, so can be used to create arbitrarily long ladders with zero or one chord per square.}

Third, \sknov{and related to our previous point}, can our results be (elegantly) generalized to \sknov{significantly} more complex recursive, \sknov{low-pathwidth} structures than ladders? If our results are generalized to more general structures, the \sknov{algorithmic} question arises of how to implement these results as reduction rules: this requires efficient algorithmic recognition of ``long'' structures in order to produce the ``short'' variant. \sknov{Already for ladders the question arises of how far one can improve upon the trivial $O(n^4)$-time algorithm for identifying such a structure, obtained by simply guessing the endpoints of the ladder.}

\sknov{Finally, we note that display graphs, which originally inspired this research, can also be constructed from three or more trees \cite{bryant2006compatibility}. The question thus arises of what type of recursive, ladder-like structures are induced in the display graph by common chains in three or more phylogenetic trees, and whether the results in this article can be easily extended to this more general situation.}

\medskip

\noindent\textbf{Acknowledgements.} We thank Hans Bodlaender and Bart Jansen  for insightful feedback. We also thank the members of our department for useful discussions.
%

%%%%%%%%%%%%%%%%%%%%%%%%%%%%%%%%%%%%%%%%%%%%%%%%%%%%%%%%%%%%%%%%%%%%%%%%%%%%%%%%%%%%%%%%%%%%%%%%%%%%%%%%%%%%%%%%
%%%   References
%%%%%%%%%%%%%%%%%%%%%%%%%%%%%%%%%%%%%%%%%%%%%%%%%%%%%%%%%%%%%%%%%%%%%%%%%%%%%%%%%%%%%%%%%%%%%%%%%%%%%%%%%%%%%%%%

{
\bibliographystyle{splncs04}
\bibliography{afterALMOB}
}

%%%%%%%%%%%%%%%%%%%%%%%%%%%%%%%%%%%%%%%%%%%%%%%%%%%%%%%%%%%%%%%%%%%%%%%%%%%%%%%%%%%%%%%%%%%%%%%%%%%%%%%%%%%%%%%%
%%%   Appendix
%%%%%%%%%%%%%%%%%%%%%%%%%%%%%%%%%%%%%%%%%%%%%%%%%%%%%%%%%%%%%%%%%%%%%%%%%%%%%%%%%%%%%%%%%%%%%%%%%%%%%%%%%%%%%%%%

\end{document}